\documentclass{proc-l}

\usepackage{amsmath}
\usepackage{amsfonts}
\usepackage{amsthm}
\input xy
\xyoption{all}

\newtheorem{theorem}{Theorem}[section]
\newtheorem{lemma}[theorem]{Lemma}
\newtheorem{proposition}[theorem]{Proposition}

\theoremstyle{definition}
\newtheorem{definition}[theorem]{Definition}

\theoremstyle{remark}

\numberwithin{equation}{section}

\def\R{\mathbb R}

\def\d{{\rm d}}
\def\qqfa{\quad{\rm for\ all}\quad}

\def\dist{{\rm dist}}

\def\A{{\rm A}}

\begin{document}

\title{Embedding of global attractors and their dynamics}

\author{Eleonora Pinto de Moura}
\address{Mathematical Institute, University of Warwick, Coventry, CV4 7AL, UK}
\curraddr{} \email{e.pinto-de-moura@warwick.ac.uk}
\thanks{EPM is sponsored by CAPES and would like to thank CAPES for all their support during her PhD}

\author{James C. Robinson}
\address{Mathematical Institute, University of Warwick, Coventry, CV4 7AL UK}
\curraddr{} \email{j.c.robinson@warwick.ac.uk}
\thanks{JCR and JJSG are supported by an EPSRC Leadership Fellowship EP/G007470/1.}

\author{J. J. S\'anchez-Gabites}
\address{Mathematical Institute, University of Warwick, Coventry, CV4 7AL UK}
\curraddr{} \email{j.j.sanchez-gabites@warwick.ac.uk}


\date{}

\dedicatory{}


\begin{abstract}
Using shape theory and the concept of cellularity, we show that if
$\mathcal A$ is the global attractor associated with a dissipative
partial differential equation in a real Hilbert space $H$ and the
set $\mathcal A-\mathcal A$ has finite Assouad dimension $d$, then
there is an ordinary differential equation in $\R^{m+1}$, with
$m>d$, that has unique solutions and reproduces the dynamics on
$\mathcal A$. Moreover, the dynamical system generated by this new
ordinary differential equation has a global attractor $\mathcal X$
arbitrarily close to $L\mathcal A$, where $L$ is a homeomorphism
from $\mathcal{A}$ into $\mathbb{R}^{m+1}$.
\end{abstract}

\maketitle

\section{Introduction}

In this paper we discuss the problem of finding a finite-dimensional
description of the asymptotic dynamics of dissipative partial
differential equations
\begin{equation}\label{eqG}
\frac{\d u}{\d t}= \mathcal G(u), \quad u \in H,
\end{equation}
where $H$ is a real separable Hilbert space with norm $\|\cdot\|$.
The evolution of the dynamical system generated by such an equation
is described by a continuous semigroup $\{S(t)\}_{t\ge0}$ of
solution operators defined by
$$S(t)u_0=u(t;u_0), \qqfa t \ge 0,$$
where $u(t;u_0)$ is the solution of the equation with initial
condition $u_0$.

Much of the long-term behaviour of the solutions of partial
differential equations can in many cases be described by global
attractors (see Langa and Robinson (1999), for example).

 \begin{definition}\label{attractor}
Let $H$ be a Hilbert space, and let $S(t)$ be a continuous semigroup
defined on $H$. A {\em global attractor} $\mathcal A \subset H$ is a
compact invariant set, i.e.
\begin{equation*}
S(t)\mathcal A=\mathcal A \qqfa t\ge0,
\end{equation*}
that attracts all bounded sets, i.e.
\begin{equation}\label{dist SX A}
{\rm dist}(S(t)B, \mathcal A)\longrightarrow 0 \quad {\rm as} \quad
t\longrightarrow\infty,
\end{equation}
for any bounded set $B \subset H$. If a global attractor $\mathcal
A$ exists, then it is unique.
 \end{definition}
 \noindent The distance in \eqref{dist SX A} is Hausdorff
semidistance between two non-empty subsets $X,Y \subset H$,
$$\dist (X,Y)= \sup_{x\in X}\inf_{y \in Y}\|x-y\|.$$

Although defined on an infinite-dimensional space, many dissipative
partial differential equations possess finite-dimensional global
attractors. This is the case, for instance, for the
Kuramoto-Sivashinsky equation and the 2D Navier-Stokes equations
(see Constantin and Foias (1988), Eden et al.\ (1994), Teman (1997),
Robinson (2001) and Chepyzhov and Vishik (2002) for a more detailed
study). It is therefore natural to seek a finite-dimensional system
of ordinary differential equations in some $\R^m$
\begin{equation}\label{ODE_IMP}
 \frac{{\rm d}x}{{\rm d}t}=\mathcal F(x)
\end{equation}
whose asymptotic behaviour reproduces that of the original equation.
Ideally,
\begin{itemize}
\item[(i)] the attractor $\mathcal{A}$ would be embedded in
$\mathbb{R}^m$ via some homeomorphism $L : \mathcal{A}
\longrightarrow L \mathcal{A} \subseteq \mathbb{R}^m$,
\item[(ii)] the dynamics of \eqref{ODE_IMP} on $L
\mathcal{A}$ would reproduce those of \eqref{eqG} on $\mathcal{A}$,
i.e. $\mathcal F(x)=L\mathcal G (L^{-1}x)$, for every $x \in L
\mathcal A$, and
\item[(iii)] $L \mathcal{A}$ would be the global attractor for \eqref{ODE_IMP}.
\end{itemize}

The existence of such a system of ordinary differential equations
has only been proved for certain dissipative equations that possess
an inertial manifold. Introduced by Foias et al.\ (1985), inertial
manifolds are positively invariant finite-dimensional Lipschitz
manifolds that contain the global attractor and attract all
trajectories at an exponential rate (see Constantin and Foias
(1988), Constantin et al.\ (1989), Foias et al.\ (1988a), Foias et
al.\ (1988b), Teman (1997), for more details). Foias et al.\ (1985)
showed that if a certain spectral gap condition holds then the
system possesses an inertial manifold. Unfortunately this condition
is very restrictive and there are many equations, such as the 2D
Navier-Stokes equations, for which it is not satisfied. Thus it is
desirable to adopt alternative approaches to the problem described
above.

Following the approach pioneered by Eden et al.\ (1994), the main
result of this paper is the following.

\begin{theorem}\label{prop:def} Suppose that the dissipative
partial differential equation
\begin{equation}\tag{\ref{eqG}}
\frac{\d u}{\d t}= \mathcal G(u), \quad u \in H,
\end{equation}
has a global attractor $\mathcal{A}$ such that
\[d := {\rm dim}_A(\mathcal{A} - \mathcal{A}) < \infty,\]
where ${\rm dim}_A$ denotes Assouad dimension. Assume that
$\mathcal{G}$ is Lipschitz continuous on $\mathcal{A}$. Then, for
any $m> \max \{d+1,6\}$ and any prescribed $\varepsilon> 0$, there
exist a system of ordinary differential equations
\begin{equation}\tag{\ref{ODE_IMP}}
 \frac{{\rm d}x}{{\rm d}t}=\mathcal F(x)
\end{equation}
 in $\R^{m}$ and a bounded linear map $L:H \longrightarrow \R^{m}$ such that:
\begin{enumerate}
    \item[\textit{1.}] the ODE \eqref{ODE_IMP} has unique solutions,
    \item[\textit{2.}] the restriction $L |_{\mathcal{A}} : \mathcal{A} \longrightarrow
    L \mathcal{A}$ is an embedding whose image $L \mathcal{A}$ is invariant under the dynamics of \eqref{ODE_IMP},
    \item[\textit{3.}] for every solution $u(t)$ of \eqref{eqG} on the attractor $\mathcal{A}$
    there exists a unique solution $x(t)$ of \eqref{ODE_IMP} such that \[u(t) = L^{-1}(x(t)),\]
    \item[\textit{4.}] the ODE \eqref{ODE_IMP} has a global attractor $\mathcal{X}$ that contains
    $L \mathcal{A}$ and is contained in the $\varepsilon$--neighbourhood of $L\mathcal{A}$,
    i.e.\ $\dist_{\rm H}(\mathcal X, L \mathcal A) \le \varepsilon$.
\end{enumerate}
\end{theorem}
\noindent We recall that the Hausdorff distance between two
non-empty subsets $X,Y \subset H$ is defined by
    $\dist_{\rm H}(X,Y)=\max\big(\dist (X,Y),\dist (Y,X)\big).$

Although item {\it 4.}\ is not ideal, we do obtain uniqueness of
solutions which is certainly desirable. The construction in Eden et
al.\ (1994), for example, has the projection of $\mathcal A$ as a
global attractor, but the finite-dimensional system of ODEs obtained
lacks uniqueness (in fact $\mathcal F$ is not even continuous).

The assumption that $\mathcal{G}$ is Lipschitz continuous on
$\mathcal{A}$ is strong - probably too strong -, but for particular
cases one can obtain some information about the smoothness of the
vector field $\mathcal G$ (see Romanov (2000) and Pinto de Moura and
Robinson (2010b), for example).

\subsection*{Structure of the paper}

The proof of Theorem \ref{prop:def} is a blend of analytical and
topological techniques, and splits naturally into the following
steps:
\begin{itemize}
    \item the existence of a linear embedding $L|_{\mathcal{A}} : \mathcal{A}
    \longrightarrow \mathbb{R}^m$ with a sufficiently regular inverse
    (namely, Lipschitz with a logarithmic correction),
    \item the construction of a system of ODEs in $\R^m$ that reproduces the dynamics of
    $\mathcal{A}$ in $L \mathcal{A}$, using the
    regularity of ${(L|_{\mathcal{A}})}^{-1}$ to guarantee that it has unique solutions,
    \item the existence of a system of ODEs in $\R^m$
    that, perhaps after replacing $L$ by a different (but related)
    linear embedding $L'$,  has $L'\mathcal{A}$ as a global attractor consisting entirely of fixed
    points and
    \item a suitable combination of the systems of ODEs constructed
    in the previous two steps that satisfies the conclusion of Theorem \ref{prop:def}.
\end{itemize}

The first step has already been dealt with in the mathematical
literature and will be addressed in Section \ref{sec:primera}, where
we limit ourselves to a discussion of the role played by the Assouad
dimension, i.e.\ the hypothesis ${\rm dim}_A(\mathcal{A} -
\mathcal{A}) < \infty$ in Theorem \ref{prop:def} and the embedding
theorem that we will be using. The second step is the content of
Proposition \ref{prop1} and closes Section \ref{sec:primera}.

In Section \ref{sec:segunda} we change gear and use topological
techniques to provide a proof of step three. The purely topological
arguments are contained in Propositions \ref{inv_cell} and
\ref{prop:top}, whereas Lemma \ref{phi3} and Proposition
\ref{prop:aux} provide the link with differential equations.
Although Lemma \ref{phi3} is less general than the analogous result
in G\"unther (1995), our proof is significantly simpler as it does
not involve piecewise linear topology. Finally, Section
\ref{sec:ultima} brings together previous results to prove Theorem
\ref{prop:def}.

We keep the notation introduced so far for the rest of the paper.

\section{Embedding the dynamics on $\mathcal{A}$ into Euclidean space \label{sec:primera}}

We first need to find an embedding of $\mathcal{A}$ into a
finite-dimensional Euclidean space. Recall that an embedding $L :
\mathcal{A} \longrightarrow \mathbb{R}^m$ is a map that is a
homeomorphism onto its image. This is a well known topological
problem which was solved in the first half of the past century (see
Hurewicz and Wallman, 1948), but in our case we need $L$ to be
``sufficiently regular'' as described below, and more care is
needed.

Since \eqref{ODE_IMP} has to reproduce the dynamics on
$\mathcal{A}$, its right hand side $f(x)$ has to bear a close
relation with $\mathcal{G}$ on the image $L \mathcal{A}$ of
$\mathcal{A}$: essentially it needs to be $L \mathcal{G} L^{-1}$. To
guarantee uniqueness of solutions for \eqref{ODE_IMP} some
regularity has to be required on $L \mathcal{G} L^{-1}$; the
standard one is Lipschitz continuity. Since $\mathcal{G}$ was
already assumed to be Lipschitz continuous only $L$ and $L^{-1}$
need to be taken care of.

Ma\~n\'e (1981) proved that if the Hausdorff dimension of the set
$\mathcal A-\mathcal A$ of differences between elements of $\mathcal
A$ is finite, then a generic projection $L$ of $H$ onto a subspace
of sufficiently high dimension is injective on $\mathcal A$. Since
projections onto finite-dimensional spaces are linear and
continuous, they are Lipschitz, which solves the problem of the
regularity of $L$. However, the condition on the Hausdorff dimension
of $\mathcal A-\mathcal A$ is not sufficient to guarantee any
regularity for ${(L|_{\mathcal{A}})}^{-1}$ (see Robinson, 2009).

Suppose for a moment that $L^{-1}$ was also required to be Lipschitz
restricted to $\mathcal A$, so that $L$ would be bi-Lipschitz. That
is, there would exist a constant $C > 0$  such that \[\frac{1}{C}
\|u - v\| \le |L(u) - L(v)| \le C \| u - v \| \qqfa u,v \in \mathcal
A,\] where $|\cdot|$ denotes some norm in $\mathbb{R}^m$. Assouad
(1983) introduced a dimension, the Assouad dimension $\dim_{\A}$
(whose definition is recalled in the following paragraph), that is
invariant under bi-Lipschitz mappings and is finite for subsets of
Euclidean space. Thus if $\mathcal{A}$ is to be embedded in a
bi-Lipschitz way into $\mathbb{R}^{m}$ we must have
$\dim_{\A}(\mathcal{A}) < \infty$.

A metric space $(X,d)$ is said to be {\em $(M,s)$-homogeneous} (or
simply {\em homogeneous}) if any ball of radius $r$ can be covered
by at most $M(r/\rho)^s$ smaller balls of radius $\rho$, for some
$M\ge 1$ and $s\ge 0$. The {\em Assouad dimension} of $X$,
$\dim_{\A}(X)$, is the infimum of all $s$ such that $(X,d)$ is
$(M,s)$-homogeneous, for some $M\ge 1$ (of course, if $X$ is not
$(M,s)$ homogeneous for any $M$ and $s$, then we define
$\dim_{\A}(X) = \infty$). Olson (2002) proved that if the
intersection with $X$ of any ball of radius $r$ can be covered by at
most $K$ balls of radius $r/2$, where $K$ is independent of $r$,
then $X$ has finite Assouad dimension. This is called the
\emph{doubling property}. For more details, see Luukkainen (1998)
and Olson (2002).

Olson and Robinson (2010) proved that, if $\dim_{\A}(\mathcal{A} -
\mathcal{A}) < \infty$, then there exists a bi-Lipschitz embedding
of $\mathcal A$ into an Euclidean space except for a logarithmic
correction term. Inspired by Olson (2002) and Olson and Robinson
(2010), Robinson (2010) proved the following embedding result that
improves the exponent of the logarithmic correction term:

\begin{theorem}[Robinson, 2010]\label{Rb10}
Let $\mathcal A$ be a compact subset of a real Hilbert space $H$
such that $\dim_{\A}(\mathcal A-\mathcal A)<s<m$. If
\begin{equation}\label{newgamma}
 \gamma> \frac{2+m}{2(m-s)},
\end{equation}
then there exists a prevalent set \footnote{The term `prevalence'
was coined by Hunt et al.\ (1992) and generalizes the notion of
`Lebesgue almost every' from finite to infinite-dimensional spaces.
The same notion was, essentially, used earlier by Christensen (1973)
in a study of the differentiability of Lipschitz mappings between
infinite-dimensional spaces. Let $V$ be a normed linear space. A
Borel subset $S\subset V$ is {\em prevalent} if there exists a
compactly supported probability measure $\mu$ such that
$\mu(S+x)=1$, for all $x\in V$. In particular, if $S$ is prevalent
then $S$ is dense in $V$.} of linear maps $L:H \longrightarrow \R^m$
that are injective on $X$ and $\gamma$-almost bi-Lipschitz, i.e.
there exist $\delta_{L}>0$, $C_{L}>0$ such that
\begin{equation}\label{Rb10embed}
  \frac{1}{C_{L}}\frac{\|u-v\|}{(-\log\|u-v\|)^{\gamma}} \le
  |L(u)-L(v)| \le C_{L}\|u-v\|,
\end{equation}
for all $u,v \in \mathcal A$ with $\|u-v\| \le \delta_{L}$.
\end{theorem}
\noindent Note that for any $\gamma>1/2$ we can choose $m$ large
enough to obtain a $\gamma$-almost bi-Lipschitz embedding into
$\R^m$. Pinto de Moura and Robinson (2010a) presented an example of
an orthogonal sequence in a Hilbert space $H$ that shows that this
bound on the logarithmic exponent $\gamma$ in Theorem \ref{Rb10} is
sharp as $m \longrightarrow \infty$.

Although reasonable, the hypothesis $\dim_{\A}(\mathcal{A} -
\mathcal{A})< \infty$ is quite restrictive, since there are no
methods available to find a bound for the Assouad dimension of
global attractors associated with dissipative equations. And, even
then, $\dim_{\A}(\mathcal{A})< \infty$ still does not imply that
$\dim_{\A}(\mathcal{A}-\mathcal{A})<\infty$ (see Olson (2002) for
details).
Moreover, only evolution equations that possess inertial manifolds
are known to satisfy this assumption and, in this case, a
finite-dimensional systems of ODEs that reproduce the behavior on
the $\mathcal A$ is already known to exist.
Nevertheless, we will assume that $\dim_{\A}(\mathcal{A} -
\mathcal{A})< \infty$ in order to study its consequences.

To conclude this section we use Theorem \ref{Rb10} above to
construct a system of ordinary differential equations with unique
solutions that reproduces the dynamics on $\mathcal A$, under the
assumptions of Theorem \ref{prop:def}.

\begin{proposition}\label{prop1}
Under the hypotheses of Theorem \ref{prop:def} and with the same
notation, for any $m > d$ there exist a system of ODEs in $\R^m$
\begin{equation}\label{ODE_R^m}
 \frac{{\rm d}x}{{\rm d}t}=g(x)
\end{equation}
and a bounded linear map $L : H \longrightarrow \R^m$ such that:
\begin{enumerate}
    \item[\textit{1.}] the function $g : \R^m \longrightarrow \R^m$ is bounded and Lipschitz
    except for a logarithmic correction,
    \item[\textit{2.}] the ODE \eqref{ODE_R^m} has unique solutions,
    \item[\textit{3.}] the restriction $L|_{\mathcal{A}} : \mathcal{A}
    \longrightarrow L \mathcal{A}$ is an embedding whose image
    is invariant under \eqref{ODE_R^m},
    \item[\textit{4.}] for every solution $u(t)$ of \eqref{eqG} on the
    attractor $\mathcal A$ there exists a unique solution $x(t)$
    of \eqref{ODE_R^m} such that
\begin{equation}
 u(t)=L^{-1}(x(t)).
\end{equation}
\end{enumerate}
\end{proposition}
\begin{proof}
It follows from Theorem \ref{Rb10} that there exists a bounded
linear map $L$ from $H$ into $\R^m$, that is injective on $\mathcal
A$ and has a Lipschitz continuous inverse on $L\mathcal A$ except
for a logarithmic correction term with logarithmic exponent
$\gamma$.

If $x(t)=Lu(t)$, where $u(t)\in \mathcal A$, then the embedded
vector field on $L\mathcal A$ is given by
\begin{equation*}
\frac{{\rm d}x}{{\rm d}t}=L\mathcal GL^{-1}(x), \ x\in L\mathcal A.
\end{equation*}
The function $g_1: L\mathcal A \longrightarrow \R^m$ such that $
 g_1(x)=L\mathcal{G}(L^{-1}(x))
$ is certainly continuous and bounded, since $L\mathcal A$ is
compact.

Next we shall consider the modulus of continuity of $g_1$. Given 
$u,v \in H$, define $Lu=x$ and $Lv=y$. It follows from Theorem
\ref{Rb10} that
\begin{equation*}
  |x-y|\ge \frac{1}{C_L}\frac{\|L^{-1}x-L^{-1}y\|}{\Big(-\log\big(\|L^{-1}x-L^{-1}y\|\big)\Big)^{\gamma}}
  \end{equation*}
Consequently, since $|Lu-Lv|\le C_L\|u-v\|$, for every $x,y \in
L\mathcal A$,
\begin{eqnarray*}
 \|L^{-1}x-L^{-1}y\| &\le& C_L \Big(-\log\big(\|L^{-1}x-L^{-1}y\|\big)\Big)^{\gamma}|x-y| \\
 &\le& C_L\Bigg (
 \log\bigg(\frac{C_L}{|x-y|} \bigg) \Bigg
 )^{\gamma}|x-y|\le C_L \ f_1(|x-y|),\\
\end{eqnarray*}
where
\begin{equation}\label{logterm}
  f_1(|x|)=|x|\Bigg (
 \log\bigg(\frac{C_L}{|x|} \bigg) \Bigg )^{\gamma}.
\end{equation}
Since we assumed that $\mathcal G$ is Lipschitz continuous, it
follows that
\begin{equation*}
  |g_1(x)- g_1(y)|\le C_L K \|L\|_{{\rm op}} f_1(|x-y|).
\end{equation*} Hence $g_1$ is Lipschitz continuous except for a logarithmic
correction term. The modulus of continuity $\omega$ of $g_1$ is
therefore the convex continuous function defined by
\begin{equation*}
  \omega(r)=C_L K \|L\|_{{\rm op}} f_1(r)= C_0 r \Big ( \log \big(C_L/r \big) \Big
  )^{\gamma}, \quad{\rm for \ } r\ge0,
\end{equation*}
 where $C_0=C_L K \|L\|_{{\rm op}}$ is a constant.

One can now use the extension theorem due to Mc Shane (1934) (see
also Stein, 1982) to extend the function $g_1$ to a function $g:\R^m
\longrightarrow \R^m$ that is Lipschitz continuous except for a
logarithmic correction term such that
\begin{equation}\label{Lip}
  |g(x)-g(y)|\le M\omega(|x-y|),
\end{equation}
for some $M>0$. It follows from \eqref{Lip} that there exists a
$T>0$ such that the initial value problem
\begin{equation}\label{IVP}
  \frac{{\rm d}x}{{\rm d}t}=g(x), \quad x(0)=x_0
\end{equation}
has at least one solution on $[0,T]$.

Now assume that $x(t)$ and $y(t)$ are solutions of \eqref{IVP} with
initial conditions $x(0)=x_0$ and $y(0)=y_0$, respectively. Let $
r(t)=|x(t)-y(t)|. $
Since the modulus of continuity $\omega(r)$ of $g$ is continuous for
$r \ge 0 $, convex and verifies
\begin{equation}\label{int_infty}
 \int_{0}^{1}\frac{{\rm d}r}{\omega(r)}=\int_{\ln (C_L)}^{\infty}s^{-\gamma}{\rm d}s
 =+\infty, \quad {\rm for \ } 0<\gamma \le 1,
 \end{equation}
we can use Osgood's Criterion (see Hartman (1964), for example) to
show that \eqref{IVP} has at most one solution on any interval $[0,
T]$, if the exponent $\gamma$ of the logarithmic term in
\eqref{logterm} is no larger than one. Since $g$ is continuous and
bounded from $\R^m$ into $\R^m$, it follows that any solution of the
initial value problem \eqref{IVP} exists for all time. Therefore the
solution of \eqref{IVP} through $x_0=Lu_0$ with $u_0\in \mathcal A$
can be uniquely given by
\begin{equation*}
  x(t)=Lu(t).
\end{equation*}
\end{proof}

\section{Making $L \mathcal{A}$ an attractor \label{sec:segunda}}

In the previous section we embedded $\mathcal{A}$ into some
finite-dimensional space $\mathbb{R}^m$ via a linear map $L : H
\longrightarrow \mathbb{R}^m$ and showed that there is a
differential equation \eqref{ODE_R^m} in $\mathbb{R}^m$ that has
unique solution and reproduces the dynamics of $\mathcal{A}$ on $L
\mathcal{A}$. To obtain a complete translation of the situation in
$H$ onto $\mathbb{R}^m$ we would like $L \mathcal{A}$ to be a global
attractor for \eqref{ODE_R^m}, which is not usually the case. As we
mentioned in the Introduction, we will only be able to modify
\eqref{ODE_R^m} in such a way that the new dynamical system still
reproduces the dynamics of $\mathcal{A}$ on $L \mathcal{A}$ and has
a global attractor $\mathcal{X}$ lying within any prescribed
(arbitrarily small) neighbourhood of $L\mathcal{A}$. We do not know
if one can construct a vector field such that $L \mathcal{A}$
itself, with the dynamics projected from $\mathcal{A}$, is a global
attractor.

In this section we show that $L \mathcal{A}$ can be made the global
attractor, comprised of equilibria, for an entirely new system of
ODEs in $\R^m$ \eqref{A}. Then in the next section we will use
\eqref{A} to add a correction term to \eqref{ODE_R^m} that will make
its solutions enter asymptotically any prescribed neighbourhood of
$\mathcal{A}$.

Difficulties arise because there is a topological obstruction to the
existence of the system of ODEs \eqref{A} having $L \mathcal{A}$ has
a global attractor: it is known that any global attractor in
Euclidean space has a property called {\em cellularity} (the
definition is recalled below), but nothing guarantees that $L
\mathcal{A}$ is cellular. So the first thing we will do is to
improve $L$ to a new linear map, temporarily denoted by $L'$, such
that $L' \mathcal{A}$ is indeed cellular. This is Proposition
\ref{inv_cell} (it will involve increasing the dimension $m$ of the
target space by one).
Then in Lemma \ref{phi3} we show that every cellular set in
Euclidean space is a global attractor for a system of ODEs and apply
this result to $L' \mathcal{A}$.

This section is built on ideas from Garay (1991) and G\"unther
(1995). The first paper singles out cellularity as a distinctive
property of attractors for flows and the second uses smoothing
results from piecewise linear topology to replace general flows by
flows arising from differential equations.

\subsection{Improving the embedding $L$}

We begin by recalling what cellularity means. A set $C$ is called a
{\em $m$-cell} if there exists a homeomorphism from $B_{\R^m}(1)$
onto $C$, where $B_{\R^m}(1)$ is the closed unit ball centered at
the origin in $\R^m$. A subset $X \subseteq \mathbb{R}^m$ is {\em
cellular} in $\mathbb{R}^m$ if there exists a cellular sequence for
$X$, that is, a sequence $(C_i)_{i \in \mathbb{N}} \subseteq
\mathbb{R}^m$ of $m$-cells that are neighbourhoods of $X$ in
$\mathbb{R}^m$ and such that $\bigcap_{i \in \mathbb{N}} C_i = X$.
Equivalently, $X$ is cellular if given any neighbourhood $U$ of $X$
there exists a $m$-cell $C \subseteq U$ that is a neighbourhood of
$X$.

It is interesting to bear in mind that whether a set $X$ is cellular
or not depends not only on its topological type, but also on how it
is embedded in $\mathbb{R}^m$.

\begin{proposition}\label{inv_cell} Let $\mathcal A$ be a
global attractor in $H$ and let $L : H \longrightarrow \mathbb{R}^m$
be a linear embedding. Then the map $L':H \longrightarrow \R^{m+1}$
defined by $L'u=(Lu,0)$ is a linear embedding whose image
$L'\mathcal A$ is cellular in $\mathbb{R}^{m+1}$, provided $m \geq
3$.
\end{proposition}

Due to the fact mentioned above that the cellularity of a set
depends on how it is embedded, we cannot prove Proposition
\ref{inv_cell} directly by saying that $\mathcal{A}$ is cellular
(because it is an attractor) and then $L \mathcal{A}$ is cellular
because it is homeomorphic, via $L$, to $\mathcal{A}$. We need to
use a different property of $\mathcal{A}$, which {\em is} invariant
under homeomorphisms. This is {\em shape}. Shape theory is a
weakening of homotopy theory that makes it extremely useful to deal
with complicated sets, roughly by overlooking their local structure.
The advantage for us is that if two spaces are homeomorphic, then
they have the same shape. In fact, something even stronger is true:
if two spaces have the same homotopy type, then they have the same
shape. We refer the reader to Borsuk (1975) and Marde{\u{s}}i{\'{c}}
and Segal (1982) for detailed information about shape theory, which
is becoming a powerful tool in the study of topological dynamics
(see G\"unther and Segal (1993), Sanjurjo (1995), Robinson (1999)).

\begin{proof}[Proof of Proposition \ref{inv_cell}] By Theorem 3.6 in
Kapitanski and Rodnianski (2000, p. 233) the set $\mathcal{A}$ has
the same shape as $H$. It is a standard fact that $H$ has the
homotopy type of a point, because the map $H \times [0,1] \ni (u,t)
\longrightarrow (1-t) \cdot u \in H$ provides a homotopy between the
identity ${\rm id} : H \longrightarrow H$ and the constant map $0 :
H \longrightarrow H$. Therefore $H$ has the shape of a point and
consequently so does $\mathcal{A}$. Since shape is invariant under
homeomorphisms, $L \mathcal{A}$ also has the shape of a point.
Thus\footnote{Daverman uses the concept of cell-likeness instead of
``having the shape of a point'', but both are equivalent. See
Section 15 in Daverman (1986).} by Daverman (1986, Corollary 5A,
Section 18) the set $L \mathcal{A} \times \{0\}$ is cellular in
$\mathbb{R}^{m+1}$, provided $m \geq 3$. But $L \mathcal{A} \times
\{0\}$ is precisely $L' \mathcal{A}$.
\end{proof}

As a side remark, and given that in our final result $L \mathcal{A}$
is not the global attractor but only closely approximated by global
attractors $\mathcal{X}$, one may wonder if our need for it to be
cellular is an accidental consequence of our method of proof. It is
not. Given an open neighbourhood $U$ of $L \mathcal{A}$, find a
system of ODEs that has a global attractor $\mathcal{X} \subseteq
U$. Since $L \mathcal{A}$ is invariant and $\mathcal{X}$ is a global
attractor, necessarily $L \mathcal{A} \subseteq \mathcal{X}$. The
set $\mathcal{X}$ is cellular, so there exists a cell $C \subseteq
U$ that is a neighbourhood of $\mathcal{X}$, hence of $L
\mathcal{A}$. Consequently $L \mathcal{A}$ has to be cellular.

\subsection{Cellular sets are global attractors for systems of ODEs}

Next we will show that if $X$ is a cellular subset of $\R^{m+1}$,
then there exists a system of ordinary differential equations
\eqref{A} with $X$ as its global attractor. G\"unther (1995) proved
a similar result for compact sets with the shape of a finite
polyhedron, but he did not need to control the size of the region of
attraction (whereas we want it to be all of $\mathbb{R}^{m+1}$). By
restricting ourselves to a less general setting and considering only
compacts sets with the shape of a point, we are able to give a
simpler proof that does not involve piecewise linear topology. The
difficulties arise in passing from well known topological results to
differentiable ones. Rather than using the uniqueness of
differentiable structures on $\mathbb{R}^n$ to do this (compare
Grayson and Pugh (1993, Corollary 2.6) for example) we have adopted
a different approach closer to G\"unther (1995) in spirit.

\begin{lemma}\label{phi3}
Given a cellular subset $X$ of $\mathbb{R}^{m+1}$, with $m > 5$,
there is a mapping $\phi : \mathbb{R}^{m+1} \longrightarrow
[0,+\infty)$ of class $\mathcal{C}^r$, where $r$ can be chosen to be
arbitrarily large, such that the equation
\begin{equation}
 \label{A} \dot{x} = -\nabla \phi(x)
\end{equation} has $X$ as a global attractor. Furthermore,
the mapping $\phi$ can be chosen to satisfy:
\begin{itemize}
  \item[(i)] $\phi(x) =0 \Leftrightarrow x \in X$ and
  \item[(ii)]  $\phi$ is proper, that is, $\phi^{-1}([s,t])$
  is compact for any $s < t \in \mathbb{R}$.
\end{itemize}
\end{lemma}
\noindent If Lemma \ref{phi3} holds, then $\nabla \phi(x) = 0
\Leftrightarrow x \in X$ since the zeros of $\nabla \phi(x)$ are
precisely the equilibria of \eqref{A}, of which there cannot be any
outside of $X$. Conversely, if $\phi : \mathbb{R}^{m+1}
\longrightarrow [0,+\infty)$ is any $\mathcal{C}^r$ mapping such
that $\nabla \phi(x) = 0 \Leftrightarrow x \in X$ and $\phi(x) = 0
\Leftrightarrow x \in X$, then by Lyapunov's theorem $X$ is a global
attractor for $\dot{x} = -\nabla\phi(x)$. Thus we only need to
construct such a $\phi$, which we do first on $\mathbb{R}^{m+1}
\backslash X$ and then extend to all of $\mathbb{R}^{m+1}$.

The proof gets a little involved because our cellularity hypothesis
is a purely topological notion but we want a differentiable map as
an outcome. Therefore we start with the following topological result
and then improve it to a differentiable one in Proposition
\ref{prop:aux}. The set $\mathbb{S}^m$ is the unit sphere in
$\mathbb{R}^{m+1}$, that is $\mathbb{S}^m = \{x \in \mathbb{R}^{m+1}
: \|x\| = 1\}$.

\begin{proposition} \label{prop:top} Let $X$ be a cellular subset of
$\mathbb{R}^{m+1}$. There exists a homeomorphism $h :
\mathbb{R}^{m+1} \backslash X \longrightarrow \mathbb{S}^m \times
(0,+\infty)$ such that the second coordinate of $h(x)$ converges to
zero when $x \longrightarrow X$.
\end{proposition}
\begin{proof} Let $Q$ be a ball in $\mathbb{R}^{m+1}$ centered
at the origin and big enough so that $X$ is contained in the
interior of $Q$. By Theorem 1 in Brown (1960) there exists a
continuous map $c : Q \longrightarrow Q$ that is onto, injective on
$Q \setminus X$, collapses $X$ to a single point $p$ in the interior
of $Q$ and is the identity on the boundary of $Q$. It is easy to
construct a homeomorphism of $Q$ onto itself that takes $p$ to $0$
and is the identity on the boundary, so we can assume that $p = 0$.

The properties of $c$ imply that $c|_{Q \setminus X} : Q \setminus X
\longrightarrow Q \setminus \{0\}$ is a homeomorphism and if $x
\longrightarrow X$ then $c(x) \longrightarrow 0$. Extend $c|_{Q
\setminus X}$ to all of $\mathbb{R}^{m+1} \backslash X$ by letting
it be the identity outside $Q$. Finally, \[h(x) := \left(
\frac{c(x)}{\|c(x)\|},\|c(x)\|\right)\] has the required properties.
\end{proof}

To make $h$ differentiable we require some smoothing results for
manifolds, rather than maps, which we take from Kirby and Siebenmann
(1977). Recall that a differential manifold is a topological
manifold equipped with a differential structure, that is an atlas of
coordinate charts such that the chart changes are
$\mathcal{C}^{\infty}$. A map between smooth manifolds is
$\mathcal{C}^{\infty}$ if its local expression in charts is
$\mathcal{C}^{\infty}$, and a diffeomorphism if it is invertible
with a $\mathcal{C}^{\infty}$ inverse (for more detailed definitions
we refer the reader to Kirby and Siebenmann, 1977).

\begin{proposition} \label{prop:aux} Let $X$ be a cellular
subset of $\mathbb{R}^{m+1}$, with $m \geq 5$. There exists a
mapping $\psi : \mathbb{R}^{m+1} \backslash X \longrightarrow
(0,+\infty)$ of class $\mathcal{C}^{\infty}$ such that:
\begin{itemize}
  \item[(i)] $\nabla \psi (x) \neq 0$ for every
 $x \in \mathbb{R}^{m+1} \backslash X$,
  \item[(ii)] $\psi(x) \longrightarrow 0$ when $x \longrightarrow X$
 and
  \item[(iii)]  $\psi$ is proper.
\end{itemize}
\end{proposition}
\begin{proof}
Consider the map $h$ obtained in Proposition \ref{prop:top}. We
would like $\psi$ to be the second coordinate of $h$, but this
choice would not be differentiable in general. Thus we first have to
smooth $h$ out. Let $\Sigma$ be the differentiable structure
$\mathbb{R}^{m+1} \backslash X$ inherits from $\mathbb{R}^{m+1}$ as
an open subset, and transport it via $h$ to obtain a new
differentiable structure $h \Sigma$ on $\mathbb{S}^m \times
(0,+\infty)$; clearly by construction $h : {\left(\mathbb{R}^{m+1}
\backslash X\right)}_{\Sigma} \longrightarrow {\left(\mathbb{S}^m
\times (0,+\infty)\right)}_{h \Sigma}$ is a diffeomorphism. Now by
Kirby and Siebenmann (1977, Theorem 5.1, p. 31) (and Remark 1
following that theorem) there is a diffeomorphism $g :
{\left(\mathbb{S}^m \times (0,+\infty)\right)}_{h \Sigma}
\longrightarrow {\left(\mathbb{S}^m\right)}_{\sigma} \times
(0,+\infty)$, where $\sigma$ is some suitable differentiable
structure on $\mathbb{S}^m$ (we need the hypothesis $m > 5$
precisely for this theorem to work). By Remark 1 following Kirby and
Siebenmann (1977, Theorem 5.1, p. 31) one can require, and it will
be technically convenient to do so, that
${\rm dist}\big(y,g(y)\big) \leq 1$ for every $y \in \mathbb{S}^m
\times (0,+\infty)$, where ${\rm dist}$ is the maximum of the
distances in $\mathbb{S}^m$ and $(0,+\infty)$.

The projection onto the second factor ${\rm pr}_2 :
{\left(\mathbb{S}^m\right)}_{\sigma} \times (0,+\infty)
\longrightarrow (0,+\infty)$ is obviously a $\mathcal{C}^{\infty}$
mapping (by definition of what a product differentiable structure
is) and its differential is never zero. Then define $\psi := {\rm
pr}_2 \circ g \circ h$, which makes the diagram \[\xymatrix{
{(\mathbb{R}^{m+1} \backslash X)}_{\Sigma} \ar[r]^-h \ar[rrd]_{\psi}
& {\left(\mathbb{S}^m \times (0,+\infty)\right)}_{h \Sigma} \ar[r]^g
& {\left(\mathbb{S}^m\right)}_{\sigma} \times (0,+\infty)
\ar[d]^{{\rm pr}_2} \\ & & (0,+\infty) }\] commutative. Clearly
$\psi$ is $\mathcal{C}^{\infty}$, because it is a composition of
$\mathcal{C}^{\infty}$ maps. Now we have to check that $\psi$
satisfies all the properties in the statement of the proposition:
\smallskip

({\it i}\/) It is clear that $\nabla \psi(x) \neq 0$, because $g$
and $h$ are diffeomorphisms (thus their differentials are
invertible) and ${\rm pr}_2$ satisfies $\nabla{\rm pr}_2(x) \neq 0$.
\smallskip

({\it iii}\/) It is convenient to deal with this one before ({\it
ii}\/). Let $s < t$, take a sequence $(x_i)_{i \in \mathbb{N}}
\subseteq \psi^{-1}([s,t])$ and denote by $(y_i,z_i) := g \circ
h(x_i)$. By hypothesis $((y_i, z_i))_{i \in \mathbb{N}} \subseteq
\mathbb{S}^m \times [s,t]$, which is a compact set, so the sequence
$((y_i,z_i))_{i \in \mathbb{N}}$ must have a convergent subsequence.
The pre-image of this subsequence under the homeomorphism $g \circ
h$ is a convergent subsequence of $(x_i)_{i \in \mathbb{N}}$. This
shows that $\psi^{-1}([s,t])$ is compact and $\psi$ is proper.
\smallskip

({\it ii}\/) Let $(x_i)_{i \in \mathbb{N}}$ be a sequence in
$\mathbb{R}^{m+1} \backslash X$ converging to $X$. We first show
that $(\psi(x_i))_{i \in \mathbb{N}}$ converges either to $0$ or
$+\infty$. Suppose not. Then it has some subsequence
$(\psi(x_{i_j}))_{j \in \mathbb{N}}$ that is contained in a compact
interval and, since $\psi$ is proper, $(x_{i_j})_{j \in \mathbb{N}}$
is contained in some compact subset of $\mathbb{R}^{m+1} \backslash
X$. This contradicts the fact that $(x_i)$ converges to $X$.

Since we required that $g$ moves points no more than $1$ unit, we
have ${\rm dist}(g \circ h(x_i)),h(x_i)) < 1$. Given that we chose
${\rm dist}$ as the maximum of the distances in $\mathbb{S}^m$ and
$(0,+\infty)$, this implies that \[{\rm dist}(\psi(x_i),{\rm pr}_2
\circ h(x_i)) = {\rm dist}({\rm pr}_2 \circ g \circ h(x_i),{\rm
pr}_2 \circ h(x_i)) < 1\] as well. Since $\psi(x_i)$ converges to
either $0$ or $+\infty$ and ${\rm pr}_2 \circ h(x_i) \longrightarrow
0$ as stated in Proposition \ref{prop:top}, it follows that
$\psi(x_i) \longrightarrow 0$.
\end{proof}

\begin{proof}[Proof of Lemma \ref{phi3}] We will construct
inductively a sequence of maps $\psi_k$, each $\psi_k$ of class
$\mathcal{C}^k$, such that $\phi := \psi_k$ proves the lemma for $r
= k$. As a first step extend the mapping $\psi$ given by Proposition
\ref{prop:aux} to all of $\mathbb{R}^{m+1}$ by letting it assume the
value $0$ on $X$, and call it $\psi_0$. This $\psi_0$ is continuous
but not differentiable near $X$, and we now use an argument hinted
at G\"unther (1995) to improve $\psi_0$ to $\psi_1$.

The idea is to let $\psi_1 := b \circ \psi_0$, where $b :
[0,+\infty) \longrightarrow [0,+\infty)$ is some diffeomorphism of
class $\mathcal{C}^1$ whose derivative near $0$ is sufficiently
small to overcome the ``roughness'' of $\psi_0$ near $X$. Formally,
for $x \in \mathbb{R}^{m+1} \backslash X$,
\[\frac{\partial}{\partial x_i} (b \circ \psi_0)(x) = (b' \circ
\psi_0)(x) \frac{\partial \psi_0}{\partial x_i}(x)\] and as $x
\longrightarrow X$ (and consequently $t = \psi_0(x) \longrightarrow
0$) we need $b'(t)$ to converge to $0$ faster than $\frac{\partial
\psi_0}{\partial x_i}(x)$ grows. We now show how to find such a $b$.

For any $t \in (0,+\infty)$ let $F_t := \{x \in \mathbb{R}^{m+1}
\backslash X : \psi_0(x) = t\}$ and
\[M(t) := \max_{x \in F_t} \left\{ \left| \frac{\partial
\psi_0}{\partial x_1}(x) \right|, \ldots, \left| \frac{\partial
\psi_0}{\partial x_m}(x) \right| \right\}.\]

Since each $F_t$ is compact, because $\psi_0$ is proper, $M(t)$ is
well defined. The condition $\nabla \psi_0(x) \neq 0$ for $x \in
\mathbb{R}^{m+1} \backslash X$ implies $M(t) > 0$ for every $t > 0$,
and clearly by construction $M\big(\psi_0(x)\big) \geq \left|
\frac{\partial \psi_0}{\partial x_i}(x) \right|$ for each $x \in
\mathbb{R}^{m+1} \backslash X$ and $1 \leq i \leq {m+1}$. Suppose
for a moment that we find a diffeomorphism $b$ such that $b'(t) \leq
\frac{t}{M(t)}$ for every $t > 0$. Then we have, for any $1 \leq i
\leq {m+1}$,
\[\frac{\partial}{\partial x_i} (b \circ \psi_0)(x) = (b' \circ
\psi_0)(x) \frac{\partial \psi_0}{\partial x_i}(x) \leq
\frac{\psi_0(x)}{M\big(\psi_0(x)\big)}\frac{\partial
\psi_0}{\partial x_i}(x) \leq \psi_0(x)\] which goes to $0$ as $x
\longrightarrow X$. Hence $\psi_1 := b \circ \psi_0$ is
$\mathcal{C}^1$ on $\mathbb{R}^{m+1}$, and its gradient on $X$ is
zero. It is still clearly regular on $\mathbb{R}^{m+1} \backslash X$
and goes to zero as $x \longrightarrow X$, so $\phi := \psi_1$
proves the lemma for $r = 1$.

If $\frac{t}{M(t)}$ were continuous, finding $b(t)$ would be a
simple matter: just take the primitive of $\frac{t}{M(t)}$ that
sends $0$ to $0$. With a little Morse theory it can be shown that
this is indeed the case; to avoid it we adopt a more elementary
approach. We begin with the following
\medskip

{\it Claim.} $M(t)$ is upper semicontinuous. That is, for each $s
\in \mathbb{R}$, the set $\{t \in (0,+\infty) : M(t) < s \}$ is
open.
\begin{proof} Fix $t_0 \in \mathbb{R}$ and $s \in
\mathbb{R}$ such that $M(t_0) < s$. We have to prove that for $t$
close enough to $t_0$ the inequality $M(t) < s$ holds.

At each point $x \in F_{t_0}$ one has $\left| \frac{\partial
\psi_0}{\partial x_i}(x) \right| < s$ for all $1 \leq i \leq {m+1}$
so, by continuity, there exists a neighbourhood $U_x$ of $x$ in
$\mathbb{R}^{m+1} \backslash X$ such that $\left| \frac{\partial
\psi_0}{\partial x_i}(y) \right| < s$ for all $y \in U_x$ and $1
\leq i \leq {m+1}$. The set $U := \bigcup_{x \in F_{t_0}} U_x$ is a
neighbourhood of $F_{t_0}$ in $\mathbb{R}^{m+1} \backslash X$. Now
clearly $F_{t_0} = \bigcap_{\varepsilon > 0}
F_{[t_0-\varepsilon,t_0+\varepsilon]}$, where
$F_{[t_0-\varepsilon,t_0+\varepsilon]} := \{y \in \mathbb{R}^{m+1}
\backslash X : \psi_0(y) \in [t_0-\varepsilon,t_0+\varepsilon]\}$.
Again because $\psi_0$ is proper, each
$F_{[t_0-\varepsilon,t_0+\varepsilon]}$ is compact, so there exists
$\varepsilon > 0$ such that $F_{[t_0-\varepsilon,t_0+\varepsilon]}
\subseteq U$. But then for $t \in [t_0 - \varepsilon,t_0 +
\varepsilon]$ we have $M(t_0) < s$, as it was to be proved.
\end{proof}

We can now find $b$. Since $M(t)$ is upper semicontinuous so is
$\frac{M(t)}{t}$, and consequently $\frac{t}{M(t)}$ is lower
semicontinuous. By a classical insertion theorem (Dowker, 1951,
Theorem 4, p. 222) it follows that there exists a continuous mapping
$0 < c(t) < \frac{t}{M(t)}$. Taking for $b(t)$ the primitive of
$c(t)$ that sends $0$ to $0$ we are finished.
\medskip

This argument can easily be adapted to provide the inductive step in
the construction of $\psi_{k+1}$ from $\psi_k$. We again let
$\psi_{k+1} := b \circ \psi_k$ for a suitable $\mathcal{C}^{k+1}$
diffemorphism $b : [0,+\infty) \longrightarrow [0,+\infty)$, but now
there are conditions to be placed on the rate at which $b^{(l)}(t)
\longrightarrow 0$ as $t \longrightarrow 0$ for every $0 \leq l \leq
k+1$. Indeed, for any multi-index $\alpha$ with $|\alpha|=k+1$ we
have
$$\frac{\partial^{\alpha}\psi_{k+1}}{\partial x^{\alpha}}
=b' \circ \psi_k \frac{\partial^{\alpha}\psi_k}{\partial
x^{\alpha}}+P\left(\frac{\partial^{\beta}\psi_k}{\partial
x^{\beta}},b^{(l)}\right)$$ on $\mathbb{R}^{m+1} \backslash X$,
where $P$ is a polynomial in partial derivatives of $\psi_k$ of
order $\leq k$ and derivatives of $b$ of order $l \le k+1$. Hence we
now need to choose $b$ subject to the conditions $b^{(l)}(0) = 0$
for every $l \leq k+1$ and $b' \circ \psi_k(x)
\frac{\partial^{\alpha}\psi_k}{\partial x^{\alpha}}(x)
\longrightarrow 0$ for $|\alpha|=k+1$ and $x \longrightarrow X$. The
first one is easy to achive; for the second one just re-read the
proof from the beginning letting \[M(t) :=  \max_{{x \in
F_t}\atop{|\alpha|=k+ 1}} \left\{ \left| \frac{\partial^{\alpha}
\psi_k}{\partial x^{\alpha}}(x) \right| \right\}.\]

\end{proof}

\section{Proof of Theorem \ref{prop:def} \label{sec:ultima}}

In this final section we assemble the results from the previous two
sections to obtain a system of ODEs \eqref{ODE_IMP} that reproduces
on $L \mathcal{A}$ the dynamics on $\mathcal{A}$ and has a global
attractor $\mathcal{X}$ as close to $L \mathcal{A}$ as required.

\begin{proof}[Proof of Theorem \ref{prop:def}] Use Proposition
\ref{inv_cell} to replace the mapping $L$ obtained in Proposition
\ref{prop1} by a new one $L' : H \longrightarrow \mathbb{R}^{m+1}$
with the additional property that its image is cellular. To keep
notation simple we rename $L'$ as $L$ and $m+1$ as $m$.

Use Lemma \ref{phi3} to obtain a $\mathcal{C}^r$ mapping $\phi :
\mathbb{R}^m \longrightarrow [0,+\infty)$ such that $L \mathcal{A}$
is a global attractor for $\dot{x} = -\nabla \phi$. Denote by
$B_{\varepsilon}(L \mathcal{A})$ the $\varepsilon$-neighbourhood of
$L \mathcal{A}$ in $\mathbb{R}^m$. Since $\phi$ is proper, there
exists $\delta > 0$ such that $P := \{x \in \mathbb{R}^{m} : \phi(x)
\leq \delta\} \subseteq B_{\varepsilon}(L\mathcal{A})$. Finally, let
$\theta : \mathbb{R}^{m} \longrightarrow [0,1]$ be a
$\mathcal{C}^{\infty}$ cut-off function such that $\theta \equiv 1$
on $L\mathcal{A}$ and $\theta \equiv 0$ outside of $P$. Take the
mapping $g$ obtained in Proposition \ref{prop1} and multiply it by
$\theta$ to make it zero outside of $P$. We shall call $f := \theta
g$; clearly $\dot{x} = f(x)$ still reproduces the dynamics of
$\mathcal{A}$ on $L \mathcal{A}$.

Now consider equations \eqref{A} and \eqref{ODE_IMP1}
\begin{align}
\dot{x} &= -\nabla \phi(x) \notag \\
\dot{x} &= f(x) - \nabla \phi(x) \label{ODE_IMP1}
\end{align}
Observe that the right hand sides of \eqref{A} and \eqref{ODE_IMP1}
coincide for $x \not\in P$. Therefore, since $\mathbb{R}^{m}
\backslash P$ is negatively invariant for \eqref{A}, it is also
negatively invariant for \eqref{ODE_IMP1} and it follows that $P$ is
positively invariant for \eqref{ODE_IMP1}.

The sets $\overline{P \cdot [t,+\infty)}$ are compact (being closed
subsets of $P$) and decreasing with increasing $t$. It is standard
that
\[\mathcal{X} := \bigcap_{t \geq 0} \overline{P \cdot [t,+\infty)}\]
is invariant and attracts $P$, i.e.\ given any $\delta > 0$ there
exists $T_{\delta} > 0$ such that $P \cdot [T_{\delta},+\infty)
\subseteq B_{\delta}(\mathcal{X})$ (see Ladyzhenskaya, 1991, Theorem
2.1). By construction, $\mathcal X$ is contained in
$B_{\varepsilon}(L\mathcal{A})$.
\smallskip

(1) $\mathcal{X}$ is a global attractor. Fix a bounded set $B
\subseteq \mathbb{R}^m$ and let \[C := \sup_{x \in B} \phi(x) \text{
and } c := \inf_{x \in B - P} \|\nabla \phi\|^2.\] Observe that $c >
0$ because $\nabla \phi$ only vanishes on $L \mathcal{A}$, of which
$P$ is a neighbourhood. Thus there exists $T > 0$ big enough so that
$C - c T < \delta$ holds.

We now claim that $x \cdot [T,+\infty) \subseteq P$ for any $x \in
B$. Since $P$ is positively invariant it clearly suffices to show
that $x \cdot t \in P$ for some $t \in [0,T]$. We reason by
contradiction, so assume that $x \cdot [0,T] \subseteq \mathbb{R}^m
\backslash P$. By the mean value theorem \[\phi(x\cdot T) = \phi(x)
+ \left.{\frac{{\rm d}}{{\rm d}s} \phi(x \cdot s)}\right|_{s = \xi}
T\] for some $\xi \in [0,T]$. Now \[\left.{\frac{{\rm d}}{{\rm d}s}
\phi(x \cdot s)}\right|_{s = \xi} = \langle \nabla\phi(x\cdot \xi),
\dot{x}(\xi) \rangle= -\|\nabla\phi(x \cdot \xi)\|^2 \leq -c,\]
where we have used the fact that $\dot{x}(\xi) = -\nabla\phi(x \cdot
\xi)$ because $x \cdot \xi \not\in P$ by assumption and
$\|\nabla\phi(x \cdot \xi)\|^2 \leq -c$ by the same token. With the
above equation and the fact that $\phi(x) \leq C$ because $x \in P$,
\[\phi(x \cdot T) \leq C - c T < \delta\] which is a contradiction
since then $x \cdot T \in P$ by definition.

Thus we see that $B \cdot [T,+\infty) \subseteq P$. Since given any
$\delta > 0$ there exists $T_{\delta} > 0$ such that $P \cdot
[T_{\delta},+\infty) \subseteq B_{\delta}(\mathcal{X})$, it follows
that $B \cdot [T+T_{\delta},+\infty) \subseteq P \cdot
[T_{\delta},+\infty) \subseteq B_{\delta}(\mathcal{X})$. Thus for $t
\geq T + T_{\delta}$ one has ${\rm dist}(B \cdot t,\mathcal{X}) <
\delta$. This implies that ${\rm dist}(B \cdot t,\mathcal{X})
\longrightarrow 0$ as $t \longrightarrow +\infty$.
\smallskip

(2) $\mathcal{X}$ contains $L\mathcal{A}$. Since $\nabla \phi$
vanishes on $L\mathcal{A}$ and $\theta \equiv 1$ on it,
\eqref{ODE_IMP1} reduces to $\dot{x} = g(x)$ when $x \in
L\mathcal{A}$. Thus $L \mathcal{A}$ is invariant for
\eqref{ODE_IMP1} and it is an immediate consequence of the fact that
$L \mathcal{A} \subseteq P$ and the expression for $\mathcal{X}$
that $L \mathcal{A} \subseteq \mathcal{X}$ (alternatively, since
$\mathcal{X}$ is the maximal compact invariant set in
$\mathbb{R}^{m}$, clearly $L \mathcal{A} \subseteq \mathcal{X}$).
\end{proof}

\section{Conclusion}

In this paper, we showed that if the compact $\mathcal A\subset H$
is the global attractor associated with a dissipative evolution
equation in 
$H$ such that the vector field $\mathcal G$ is Lipschitz continuous
on $\mathcal A$ and $\dim_{\A}(\mathcal A-\mathcal A)=d$, then there
is an ordinary differential equation in $\R^{m+1}$, with $m>d$, that
has unique solutions and reproduces the dynamics on $\mathcal A$.
Moreover, we proved that the dynamical system generated by this new
ordinary differential equation has a global attractor $\mathcal X$
arbitrarily close to $L\mathcal A$, where $L$ is a bounded linear
map from $H$ into $\R^{m+1}$ that is injective on $\mathcal{A}$.

Nevertheless, the existence of a system of ordinary differential
equation whose asymptotic behavior reproduces the dynamics on
$\mathcal A$ and has $L \mathcal A$ as a global attractor remains an
interesting open problem. In addition, the assumption that the
vector field $\mathcal G$ is Lipschitz continuous on the global
attractor $\mathcal A$ is quite strong and it would be interesting
to weaken it.

Finally, the results presented in this paper highlight the
importance of finding a general method to bound the Assouad
dimension of the set $\mathcal A-\mathcal A$, where $\mathcal A$ is
a global attractor associated with a partial differential equation
in $H$. However Eden et al.\ (1994, Lemma 2.1) showed that, for a
large class of dissipative equations for which the squeezing
property holds, there exists a constant $K>0$, such that the set
$S(T)[\mathcal A \cap B(x,r)]$ can be covered by $K$ balls of radius
$\theta r$, for some $T>0$. Hence, given its similarity with the
doubling property mentioned earlier, it might be possible to use the
above result to bound $\dim_{\A}(\mathcal A)$.


\bibliographystyle{amsplain}

\end{document}